\documentclass[11pt,reqno]{amsart}%
\usepackage[latin1]{inputenc}
\usepackage{amsfonts,amssymb,latexsym}
\usepackage{amsmath,amsthm}
\usepackage[Sonny]{fncychap}
\usepackage{enumerate}
\usepackage[dvips]{color,graphicx}
\usepackage{amsbsy}
\usepackage{amsfonts}
\usepackage{graphicx}
\usepackage{color}
\usepackage{amsmath}
\usepackage{amssymb}%
\newtheorem{theorem}{Theorem}[section]
\newtheorem{definition}[theorem]{Definition}

\newtheorem{lemma}[theorem]{Lemma}
\newtheorem{remark}{Remark}[section]

\numberwithin{equation}{section}
\begin{document}
\title[ ]{A distributed control problem for a three-dimensional Lagrangian averaged
Navier-Stokes-$\alpha$ model }
\author[E.J. Villamizar-Roa]{E.J. Villamizar-Roa}
\address[E.J. Villamizar-Roa]{(Corresponding Author) Universidad Industrial de Santander, Escuela de
Matem\'{a}ticas, A.A. 678, Bucaramanga, Colombia.} \email{jvillami@uis.edu.co}
\author[E. Ortega-Torres]{E. Ortega-Torres }
\address[E. Ortega-Torres]{ Departamento de Matem\'aticas, Universidad
Cat\'olica del Norte, Casilla 1280, Antofagasta-Chile.}
\email{eortega@ucn.cl}
\thanks{E. Ortega-Torres was supported by Fondecyt-Chile, Grant 1080399}
\date{\today}

\begin{abstract}
A distributed optimal control problem with final observation for a
three-dimensional Lagrange averaged Navier-Stokes-$\alpha$ model is
studied. The solvability of the optimal control problem is proved
and the first-order optimality conditions are established. Moreover,
by using the Lagrange multipliers method, an optimality system in a
weak and a strong form is derived.
\end{abstract}
\maketitle

\section{Introduction}
We are interested in the study of an optimal control problem for a
Lagrange averaged Navier-Stokes-$\alpha$ model (also known as
LANS-$\alpha$ or the viscous Camassa-Holm system). The LANS-$\alpha$
model, introduced by S. Chen, C. Foias, D.D. Holm, E. Oslon, E.S.
Titi, and S. Wynne in \cite{chen1}, is the first one to use
Lagrangian averaging to address the turbulence closure problem, that
is, the problem of capturing the physical phenomenon of turbulence
at computably low resolution. This model provides closure by
modifying the nonlinearity in the Navier-Stokes equations to stop
the cascading of turbulence at scales smaller than a certain length,
but without introducing any extra dissipation (c.f. \cite{lans2,
chen1, chen2, chen3, holm, lans4}). The mathematical model is
obtained by regularizing the 3D Navier-Stokes equations through a
filtration of the fluid motion that occurs below a certain length
scale $\alpha^{1/2};$ the length scale is the filter width derived
from inverting the Helmholtz operator $I-\alpha\Delta.$ Explicitly,
LANS-$\alpha$ model can be written in the following form:
\begin{equation}
\left\{
\begin{array}
[c]{rcl}%
&&(I-\alpha \Delta) u_t+\nu (I-\alpha \Delta)
Au+(u\cdot\nabla)\,(I-\alpha \Delta)u
-\alpha (\nabla u)^*\cdot \Delta u + \nabla p= f \mbox{ in } Q, \\
&& \nabla \cdot u =0 \mbox{ in } Q,\\
&& u=0, \quad Au=0 \mbox{ on } \Gamma \times (0,T),\\
&& u(x,0)= u_0(x) \mbox{ in } \Omega,
\end{array}
\right.  \label{eq1}
\end{equation}
where $u$ and $p$ are unknown, representing respectively, the
large-scale (or averaged) velocity and the pressure, in each point
of $Q=\Omega  \times (0,T), 0 < T <\infty.$ Here, $\Omega$ is a
bounded domain in $\mathbb{R}^3$ (with boundary $\Gamma$of class
$C^2$) where the fluid is occurring, and $(0,T)$ is a time interval.
The operator $A$ denotes the known Stokes operator. Moreover, the
right hand side $f$ is a fixed external force and $u_0$ a given
initial velocity field. The positive constant $\nu$ represents the
kinematic viscosity of the fluid.

The interest of studying the LANS-$\alpha$ models arises principally
in the approximation of many problems relating to turbulent flows
because it preserves the properties of transport for circulation and
vorticity dynamics of the Navier-Stokes equations. One of the main
reasons justifying its use is the high-computational cost that the
Navier-Stokes model requires \cite{lans2}. Notice that when
$\alpha=0$ the LANS-$\alpha$ model reduces to the classical
Navier-Stokes system. We refer \cite{lans2, chen1, chen2, chen3,
lans7, lans8, holm, lans4} and references therein, for a complete
description of the development of the LANS-$\alpha$ model, as well
as, a discussion about the physical significance, namely, in
turbulence theory.

From a mathematical point of view, several advances related the
well-posedness, long time behavior, decay rates of the velocity and
the vorticity, the connection between the solutions of the
LANS-$\alpha$ model and the 3D Navier-Stokes system and
Leray-$\alpha$ model, the existence and uniqueness of solutions for
stochastic versions, have been developed in last years, see for
instance \cite{lans9, lans2, marquez-duran, lans5, Cheskidov,
coutand, Fo-Ho-Ti, lans7, lans8, lans4, lans3} and references
therein. In particular, opposed to three dimensional Navier-Stokes
equations, for LANS-$\alpha$ model, the existence and uniqueness of
weak solutions is known (see for instance \cite{lans7}). This point
is relevant in control problems because it permits to guarantee that
the reaction of the flow produced by the action of a control is
unique.

In this paper we are interested in an optimal control problem for
the LANS-$\alpha$ model (\ref{eq1}) where the body force is regarded
as the control and a final observation is considered; in this sense
we say that it is an optimal control problem for a distributed
parameter system with final observation. More precisely, we wish to
minimize the functional
\[J(u,v)=\frac{\gamma_1}{2}\int_0^T\Vert u(t)-u_d(t)\|_{D(A)}^2
dt+\frac{\gamma_2}{2}\int_\Omega\vert u(x,T)-u_T(x)\vert^2dx
+\frac{\gamma_3}{2}\int_0^T \|v(t)\|_{2}^2 dt,\] where the velocity
field is subject to state system (\ref{eq1}) where $f$ is now
replaced by the distributed control field $v.$ The functions $u_d,
u_T$ are given and denote the desired state, and the parameters
$\gamma_1,\gamma_2,\gamma_3>0$ stand the cost coefficients for the
control. The exact mathematical formulation will be given in Section
3. We will prove the solvability of the optimal control problem and
state the first-order optimality conditions. By using the Lagrange
multipliers method we derive an optimality system in a weak and
strong formulation. To the best of our knowledge, this paper is the
first work dealing with optimal control problems where the state
variable satisfies the 3D LANS-$\alpha$ model (\ref{eq1}). However,
in the recent papers \cite{tian2,tian1} the authors studied the
problem of optimal control of the viscous Camassa-Holm equation in
one dimension. The models treated in \cite{tian2,tian1} can be
viewed as one dimensional versions of the three dimensional
LANS-$\alpha$ model.

Related to the nonstationary Navier-Stokes system, there are many
results available in the literature concerned with the study of
optimal control problems (see \cite{fursikov} and references
therein). In particular, necessary conditions for optimal control of
2D-Navier-Stokes model can be found in \cite{abergel, gunzburger1,
gunzburger2, hinze}. Necessary conditions for optimal control of 3D
Navier-Stokes were obtained in \cite{casas}.

The paper is organized as follows. In section 2 we establish the
notation to be used and recall some known results for the
LANS-$\alpha$ model. In section 3 we setting the precise optimal
control problem and prove the existence of optimal solutions. In
section 4 we derive the first-order optimality conditions, and by
using the Lagrange multipliers method we derive an optimality
system.

\section{Preliminaries}
Let $\Omega$ be a bounded domain in $\mathbb{R}^3$ with boundary
$\Gamma$ of class $C^2$. We denote by $\mathcal{D}(\Omega)$ and
$\mathcal{D}^{\prime}(\Omega)$ the space of functions of class
$C^{\infty}(\Omega)$ with compact support, and the space of
distributions on $\Omega,$ respectively. Throughout this paper we
use standard notations for Lebesgue and Sobolev spaces. In
particular, the $L^2(\Omega)$-norm and the $L^2(\Omega)$-inner
product, will be represented by $\Vert \cdot\Vert$ and
$(\cdot,\cdot),$ respectively. We consider the solenoidal Banach
spaces $H$ and $V$ defined, respectively, as the closure in
$(L^2(\Omega))^3$ and $(H^1(\Omega))^3$ of
\begin{eqnarray*}
\mathcal{V}&=&\{u \in (\mathcal{D}(\Omega))^3 : \nabla \cdot u=0
\mbox{ in } \Omega\}.
\end{eqnarray*}
Here, $\nabla \cdot u$ denotes the divergence of the field $u.$ The
norm and the inner product in $V$ will be denoted by $\Vert
u\Vert_V$ and $(\nabla u, \nabla v),$ respectively. Throughout this
paper, if $X$ is a Banach space with dual space $X'$, the duality
paring between $X'$ and $X$ will be denoted by $\langle\cdot,
\cdot\rangle_{X',X}.$ To simplify the notation, we will
use the same notation for vectorial valued and scalar valued spaces.

For $X$ Banach space, $\|\cdot\|_X$ denotes its norm and
$L^p(0,T;X)$ denotes the standard space of functions from $[0,T]$ to
$X,$ endowed with the norm
\[\|u\|_{L^p(0,T;X)}= \bigg( \int_0^T\|u\|_X^p dt\bigg)^{1/p}, \ 1\leq p <\infty,\qquad
 \|u\|_{L^\infty(0,T;X)}= \sup_{t\in (0,T)} \|u(t)\|_X.\]
In the sequel we will identify the spaces $L^p(0,T;X):= L^p(X)$ and
$L^p(0,T;L^p(\Omega)) :=L^p(Q)$.\newline We recall the following
compactness result:
\begin{lemma} (\cite{Simon})\label{lema1}
Let $B_0, B$ and $B_1$  be Banach spaces with $B_0\hookrightarrow B\hookrightarrow
B_1$ continuously and  $B_0\hookrightarrow B$ compact. For $1\leq p\leq
\infty$ and $T <\infty$ consider the Banach space
\begin{equation}\label{eq5}
W=\{ u \in L^{p} (0,T;B_0), \ u_t\in L^{1} (0,T;B_1)\}.
\end{equation}
Then $W \hookrightarrow L^p(0,T;B)$ compactly.
\end{lemma}

Let $P:L^2(\Omega) \rightarrow H$ be the Leray projector, and denote
by $A= - P\Delta$ the Stokes operator with domain
$D(A)=H^2(\Omega)\cap V$. It is known that $A$ is a self-adjoint
positive operator with compact inverse. Since $\Gamma$ is of class
$C^2$, the norms $\|A u\|$ and $\|u\|_{H^2}$ are equivalent.
For $u\in D(A)$ and $v\in L^2(\Omega)$ we define the element of
$H^{-1}(\Omega)\equiv (H^1_0(\Omega))'$ by
$$
\langle(u\cdot\nabla)v,w\rangle_{H^{-1},H^1_0}=\sum_{i,j=1}^3\langle\partial_i
v_j,u_iw_j\rangle_{H^{-1},H^1_0},\ \forall w\in H^1_0(\Omega).
$$
In particular, if $v\in H^1(\Omega),$ the definition of
$\langle(u\cdot\nabla)v,w\rangle_{H^{-1},H^1_0}$ coincides with the
definition of
$$
((u\cdot\nabla)v,w)=\sum_{i,j=1}^3\int_\Omega(u_i\partial_i
v_j)w_jdx.
$$
Let us denote by $(\nabla u)^*$ the transpose of $\nabla u.$ Thus,
if $u \in D(A)$ then $(\nabla u)^* \in H^1(\Omega) \subset
L^6(\Omega)$. Consequently, for $v\in L^2(\Omega)$ we have that
$(\nabla u)^*\cdot v \in L^{3/2}(\Omega)\subset H^{-1}(\Omega)$ with
$$
\langle(\nabla u)^*\cdot
v,w\rangle_{H^{-1},H^1_0}=\sum_{i,j=1}^3\int_\Omega(\partial_ju_i)
v_iw_jdx,\ \forall w\in H^1_0(\Omega).
$$
One can check that for $u,w \in D(A), v\in L^2(\Omega),$ the
following equality holds
\begin{equation}\label{eq7}
\langle(u\cdot \nabla) v, w\rangle_{H^{-1},H^1_0}= -\langle(\nabla
w)^*\cdot v,u\rangle_{H^{-1},H^1_0}.
\end{equation}
We consider the nonlinear operator  $ B: D(A)\times D(A)\rightarrow
D(A)'$ defined by
\begin{eqnarray}\label{eq8}
\langle B(u,v),w\rangle_{D(A)',D(A)}=\langle(u\cdot \nabla)(v-\alpha
\Delta v), w\rangle_{V',V}+\langle(\nabla u)^*\cdot (v-\alpha\Delta
v),w\rangle_{V',V}.
\end{eqnarray}
Thus, from (\ref{eq7}) we have
\begin{equation}\label{eq9}
\langle B(u,v),u \rangle_{D(A)',D(A)}= 0, \ \forall \, u,v \in D(A).
\end{equation}
Also, we have that
\begin{eqnarray*}
|\langle B(u,v),w\rangle_{D(A)',D(A)}|&\leq &C\|u\|\|\nabla v\|\|A
w\|+ C\,\alpha (\|u\|_{L^6}\|\nabla w\|_{L^3}+
\|\nabla u\|\|w\|_{L^\infty})\|\Delta v\|\\
&\leq & C\|\nabla u\|\|A v\|\|A w\|+C\alpha \|\nabla u\|\|A w\|\|A
v\| \leq  C\|\nabla u\|\|A v\|\|A w\|.
\end{eqnarray*}
Therefore,
\begin{equation}\label{eq10}
\|B(u,v)\|_{D(A)'} \leq C\,\|\nabla u\|\|A v\|\leq C\Vert
u\Vert_V\Vert v\Vert_{D(A)}, \quad \forall \, u, v \in D(A),
\end{equation}
and thus, for all $u, v \in L^\infty(V)\cap L^2(D(A))$ it holds
$B(u,v) \in  L^2(D(A)')$. 
Denoting by $ \Delta_\alpha= I-\alpha \Delta$, one gets
\[\Delta_\alpha u \in L^\infty(V') \cap L^2(H)\quad \mbox{ and }
\quad \Delta_\alpha Au \in L^2(D(A)') \quad \forall u\in
L^2(D(A))\cap L^\infty(V).\] With the above notations, the system
(\ref{eq1}) can be rewritten as
\begin{equation}
\left\{
\begin{array}
[c]{rcl}%
&&\Delta_\alpha u_t +\nu \Delta_\alpha Au+ B(u,u) + \nabla p
= f \mbox{ in } Q,\\
&&\nabla\cdot u=0  \mbox{ in } Q, \\
&& u=0, \quad Au=0 \mbox{ on } \Gamma \times (0,T),\\
&& u(x,0)= u_0(x) \mbox{ in } \Omega.
\end{array}
\right.  \label{eq11}
\end{equation}
Now we are in position to establish the definition of weak solution of Problem (\ref{eq1}) (equivalently (\ref{eq11})).
\begin{definition}(\textit{Weak solution})
For $f\in L^2(Q)$ and $u_0\in V,$ a \textit{weak solution} of the
problem (\ref{eq11}) is a field $u\in L^2(D(A)) \cap
L^\infty(V)$ with $u_t \in L^2(H)$ satisfying

\begin{equation}
\left\{
\begin{array}
[c]{rcl}%
&&\frac{d}{dt}((u, w) +\alpha(\nabla u, \nabla w))
+\nu (Au,w+\alpha  Aw)\\
&&\ \ \ \ \ \ +\langle B(u,u), w\rangle_{D(A)',D(A)} 
=(f, w), \quad \forall \, w\in D(A),\\
&&\ u(x,0)= u_0(x) \mbox{ in } \Omega,
\end{array}
\right.  \label{eq16}
\end{equation}

or equivalently,
\begin{equation}
\left\{
\begin{array}
[c]{rcl}%
&&\Delta_\alpha u_t+\nu \Delta_\alpha Au+ B(u,u)
= f \mbox{ in } D(A)', \\
&& A u=0 \mbox{ in } \Gamma \times (0,T),\\
&& u(x,0)= u_0(x) \mbox{ in } \Omega.
\end{array}
\right.  \label{eq17}
\end{equation}
\end{definition}

\begin{theorem}\label{teor2} (\textit{Existence and uniqueness of weak solution}) Assuming that
$f\in L^2(Q)$ and $u_0 \in V$, there exists a unique weak solution
of (\ref{eq11}).
\end{theorem}
\textit{Proof.} The existence of weak solutions follows from the classical Galerkin
approximations and energy estimates; the uniqueness follows from a
standard Gronwall argument (see e.g. \cite{BRV,real1,lans2,lans5,lans7,
lans3}).

\section{A distributed control problem: Existence of optimal solution}
We start by establishing the control in the system. We denote the
control by $v\in L^2(Q)$ which will be use as a source term in
(\ref{eq11})$_1$; thus the right-hand side of equality
(\ref{eq17})$_1$ will be defined as:
\[(f,w)  =(v,w), \quad \forall w\in D(A).\]
In order to specify exactly the problem, we make some
considerations. We define the Banach space
\[\mathbb{W}=\{ u\in L^2(D(A))\cap L^\infty(V)\,:\, u_t\in L^2(H)\},\]
with norm given by $$\Vert w\Vert_{\mathbb{W}}:=\max \{\Vert
u\Vert_{L^2(D(A))},\Vert u\Vert_{L^\infty(V)}, \Vert u_t\Vert_{L^2(H)}\}.$$ 

Since
$D(A)\hookrightarrow V \hookrightarrow H$, and $D(A)\hookrightarrow V $ compactly, from Lemma \ref{lema1}
we have
$\mathbb{W}\hookrightarrow L^2(V)$  compactly; furthermore, and as $D(A),V,H$ are Hilbert spaces,
$\mathbb{W} \hookrightarrow C([0,T];V)$ (cf. \cite{lions}). We also consider the subspace
$\mathbb{W}_0$ of $\mathbb{W}$ defined by
$$ \mathbb{W}_0:=\{ u\in \mathbb{W} : Au=0 \mbox{ on } \Gamma\times (0,T)\}.
$$

In order to establish the control problem, we assume the following general
hypotheses:

\begin{itemize}
\item[(H1)] The regularization parameters $\gamma_1, \gamma_2$ and  $\gamma_3,$ which measures the
cost of the control, are fixed positive numbers.

\item[(H2)] The initial data $ u_0 \in V$, the desired velocity
$u_d\in L^2(D(A))$ and the function $u_T \in H$.

\item[(H3)] The set of admissible
controls $\mathcal{U}_{ad}$ is defined by:
\begin{equation}\label{eq22}
 \mathcal{U}_{ad}=\{v \in L^2(Q): v_{a,i}(x,t) \leq v_i(x,t)\leq v_{b,i}(x,t)
\ a.e.  \mbox{ on } Q,\ i=1,2,3\},
\end{equation}
where the control constraints $v_a, v_b$ are required to be in $L^2(Q)$
with $ v_{a,i}(x,t) \leq v_{b,i}(x,t)$ a.e. on $Q.$ Notice that $\mathcal{U}_{ad}$ is a non-empty, convex and closed set
in $L^2(Q)$.
\end{itemize}

Under the hypotheses (H1)-(H3), for $(u,v)\in \mathbb{W}_0\times
\mathcal{U}_{ad}$ we define the following objective functional
\begin{equation}\label{eq23}
J(u,v)=\frac{\gamma_1}{2}\int_0^T\|u(t)-u_d(t)\|_{D(A)}^2
dt+\frac{\gamma_2}{2}\int_\Omega |u(x,T)-u_T(x)|^2 dx
+\frac{\gamma_3}{2}\int_0^T \|v(t)\|^2dt.
\end{equation}
Thus, we consider the following distributed optimal control problem:
\begin{equation}
\left \{
\begin{array}
[c]{rcl}%
\mbox{Minimize}&& J(u,v)\\
&&\ \ ({u},{v}) \in \mathbb{W}_0\times\mathcal{U}_{ad},
\end{array}
\right.\label{eq24}
\end{equation}
subject to the state equation
\begin{equation}
\left\{
\begin{array}
[c]{rcl}%
&& \Delta_\alpha u_t +\nu \Delta_\alpha Au+  B(u,u)
= v  \mbox{ in } L^2(D(A)'),\\
&& u(x,0)= u_0(x) \mbox{ in } V.
\end{array}
\right.  \label{eq25}
\end{equation}
Then, the set of admissible solutions to (\ref{eq24})-(\ref{eq25})
is defined as:
\begin{equation}\label{eq28}
 \mathcal{S}_{ad}=\{(u,v) \in \mathbb{W}_0\times \mathcal{U}_{ad}:
 J(u,v) < \infty \mbox{ and } (u,v)
\mbox{ satisfies } (\ref{eq25})\}.
\end{equation}
\subsection{Existence of solution}
We will show that the optimal control problem
(\ref{eq24})-(\ref{eq25}) has a solution.
\begin{theorem}
Under the assumptions $(H1)$-$(H4)$, there exists a solution
$(\hat{u}, \hat{v}) \in \mathcal{S}_{ad}$ to the optimal control
problem (\ref{eq24})-(\ref{eq25}).
\end{theorem}
\begin{proof} By Theorem \ref{teor2}, the pair $(u,v_a)\in \mathcal{S}_{ad}$;
thus the set $\mathcal{S}_{ad}\neq \phi$. Now, since $J$ is bounded
from below ($J(u,v) \geq 0$),  there exists a minimizing sequence
$(u^m, v^m)$ in $\mathcal{S}_{ad}$ such that
\[\lim_{m\rightarrow\infty} J(u^m,v^m)=\inf\{J(u,v): (u,v)\in\mathcal{S}_{ad}\},\]
and for all $w\in D(A)$ it holds:
\begin{equation}\label{eq29}
\langle\Delta_\alpha
u_t^m,w\rangle_{D(A)',D(A)}+\nu\langle\Delta_\alpha Au^m,
w\rangle_{D(A)',D(A)}+\langle B(u^m,u^m),w\rangle_{D(A)',D(A)} =(
v^m,w).
\end{equation}
From (\ref{eq29}), using integration by parts on $\Omega$, we have
\begin{equation}\label{eq30}
( u_t^m,w)+\alpha(\nabla u^m_t,\nabla w) +\nu(\nabla u^m, \nabla w)
+\nu \alpha( Au^m, A w)+\langle B(u^m, u^m),w\rangle=( v^m,w).
\end{equation}
Then, setting $w=u^m(t)$ in (\ref{eq30}) and taking into account
(\ref{eq9}), we get
\[
\frac12\frac{d}{dt}(\|u^m\|^2+\alpha\|\nabla u^m\|^2)+\nu \|\nabla
u^m\|^2 +\nu \alpha \|A u^m\|^2=(v^m,u^m).\] Using the H\"{o}lder
and Young inequalities it holds
\[|( v^m,u^m) | \leq \|v^m\|\|u^m\|\leq C_\nu \,\|v^m\|^2 + \frac{\nu}{2}\|\nabla u^m\|^2.\]
Therefore,
\begin{equation}\label{eq31}
\frac{d}{dt}(\|u^m\|^2+\alpha\|\nabla u^m\|^2)+\nu \|\nabla u^m\|^2
+2\nu \alpha \|A u^m\|^2\leq C_\nu \,\|v^m\|^2.
\end{equation}
Integrating (\ref{eq31}) from $0$ to $t\in [0,T]$, we obtain
\begin{eqnarray}\label{eq32}
&&\|u^m(t)\|^2+\alpha\|\nabla u^m(t)\|^2+\nu  \int_0^t (\|\nabla
u^m(s)\|^2
+\alpha\|A u^m(s)\|^2) ds \nonumber\\
&&\hspace*{1cm} \leq C\,\int_0^t \|v^m(s)\|^2 ds +
\|u_0\|^2+\alpha\|\nabla u_0\|^2.
\end{eqnarray}
Since $v^m \in \mathcal{U}_{ad}$ and $u_0\in V,$ from
(\ref{eq32}) we conclude that
\begin{eqnarray}\label{eq33}
\{u^m\}_{m\geq 1} \ \mbox{ is uniformly bounded in} \ L^\infty(V)\cap
L^2(D(A))\\
\{v^m\}_{m\geq 1} \ \mbox{ is uniformly bounded in} \ L^2(Q).\label{eq34}
\end{eqnarray}
On the other hand, from (\ref{eq29}) and by applying integration by parts on $\Omega$
we get
\[\langle\Delta_\alpha u^m_t,w\rangle_{D(A)',D(A)}=-\nu(\nabla u^m, \nabla
w) -\nu \alpha(Au^m, A w)-\langle B(u^m,
u^m),w\rangle_{D(A)',D(A)}+( v^m,w),
\] and then, by using the H\"{o}lder inequality together
inequalities (\ref{eq10}) and (\ref{eq33}), we obtain
\begin{eqnarray*}
|\langle\Delta_\alpha u^m_t,w\rangle_{D(A)',D(A)}|
&\leq & C_{\nu,\alpha}(\|\nabla u^m\|+\|Au^m\|+\|B(u^m, u^m)\|_{D(A)'}+\|v^m\|)\|w\|_{D(A)}\\
&\leq &C_{\nu,\alpha}(\|\nabla u^m\|+\|Au^m\|+\|\nabla u^m\|\|A u^m\|+\|v^m\|)\|w\|_{D(A)}\\
&\leq &C_{\nu,\alpha}(\|\nabla u^m\|+\|Au^m\|+\|v^m\|)\|w\|_{D(A)}.
\end{eqnarray*}
Since $\langle \Delta_\alpha u_t,w\rangle_{D(A)',D(A)}= \langle
u_t+\alpha A u_t,w\rangle_{D(A)',D(A)} $ for all $w\in D(A)$, the
last inequality implies
\[\|u_t^m+\alpha A u^m_t\|_{D(A)'}\leq C\,(\|\nabla u^m\|+\|Au^m\|+\|v^m\|),\]
and by using the Young inequality
\begin{equation}\label{eq35}
\|u_t^m+\alpha A u^m_t\|^2_{D(A)'} \leq C\,(\|\nabla
u^m\|^2+\|Au^m\|^2+\|v^m\|^2).
\end{equation}
By integrating (\ref{eq35}) from $0$ to $t\in [0,T]$ and taking into
account (\ref{eq33})-(\ref{eq34}) we have
\begin{equation}\label{eq36}
\int_0^t \|u_t^m(s)+\alpha A u^m_t(s)\|^2_{D(A)'}ds \leq C.
\end{equation}
Since the operator $A$ is self adjoint and positive, the following
inequality holds (see \cite{lans3})
\begin{equation}\label{eq37}
\|v\|^2_{D(A)'} \leq \|v+\alpha A v\|_{D(A)'}^2 \ \mbox{for each }\
v\in D(A)'.
\end{equation}
Then, by using triangular inequality and (\ref{eq37}), we get
\begin{equation}\label{eq38}
\|\alpha Au^m_t\|_{D(A)'}^2\leq \|u^m_t+\alpha Au^m_t\|_{D(A)'}^2+
\|u^m_t\|^2_{D(A)'} \leq C\,\|u^m_t+\alpha Au^m_t\|_{D(A)'}^2.
\end{equation}
Thus, from (\ref{eq36}) and (\ref{eq38}), we conclude that

\begin{equation}\label{eq39}
\{Au^m_t\}_{m\geq 1} \ \mbox{ is uniformly bounded in} \ L^2(D(A)'),
\end{equation}
\begin{equation}\label{eq40}
\{u^m_t\}_{m\geq 1} \  \mbox{ is uniformly bounded in} \ L^2(Q).
\end{equation}
Moreover, from
(\ref{eq33}) and (\ref{eq40}) we have
\begin{equation}\label{eq41}
\{u^m\}_{m\geq 1} \  \mbox{ is uniformly bounded in } \ \mathbb{W},
\end{equation}
with $\mathbb{W}$ compactly imbedded in $L^2(V)$ (cf. Lemma \ref{lema1}).\newline
 Then, from
(\ref{eq33}), (\ref{eq34}), (\ref{eq39}), (\ref{eq40}) and
(\ref{eq41}), there exists a subsequence, which again we denote by
$(u^m, v^m),$ converging to some limit $(\hat{u},\hat{v}) \in
\mathbb{W} \times\mathcal{U}_{ad}$ such that as $m \rightarrow
\infty$,
\begin{eqnarray*}
&& u^m \rightarrow \hat{u} \mbox{ weakly in } L^2(D(A)) \mbox{ and strongly in } L^2(V),\\
&& u^m_t \rightarrow \hat{u}_t \mbox{ weakly in } L^2(Q),\\
&& A u^m_t \rightarrow A\hat{u}_t\mbox{ weakly in } L^2((D(A)'),\\
&& v^m \rightarrow \hat{v} \mbox{ weakly in } L^2(Q).
\end{eqnarray*}
Passing to the limit as $m\rightarrow \infty$ in (\ref{eq29}), we
can obtain that $(\hat{u}, \hat{v})$ satisfies (\ref{eq25}), with $A \hat{u}=0$ on $\Gamma\times (0,T),$ and
$J(\hat{u}, \hat{v})<\infty.$
Since $u_0=u^m(0)$ in $V$ for all $m$, and $u^m\in \mathbb{W},$ it holds $u_0=\hat{u}(0)$ in $V$. Consequently $\hat{u}\in \mathbb{W}_0$ and 
thus we get that $(\hat{u},
\hat{v}) \in \mathcal{S}_{ad}.$ Then we obtain
\begin{equation}\label{eq42}
\lim_{m\rightarrow \infty} J(u^m, v^m) = \inf\{J(u,v): (u,v)
\in \mathcal{S}_{ad}\}\leq J(\hat{u}, \hat{v}).
\end{equation}
As the functional $J: \mathcal{S}_{ad} \rightarrow \mathbb{R}$ is
weakly lower semicontinuous, we have that (cf. \cite{brez})
\begin{equation}\label{eq43}
J(\hat{u}, \hat{v}) \leq \lim_{m\rightarrow \infty}\inf J(u^m, v^m).
\end{equation}
Finally, from (\ref{eq42}) and (\ref{eq43}) we conclude that
$\displaystyle J(\hat{u},\hat{v}) =\inf\{J(u,v): (u,v) \in
\mathcal{S}_{ad}\}$.
\end{proof}

\section{First order optimality conditions}
In order to derive an optimal system by using the Lagrange
multiplier method, we formulate an abstract Lagrange multiplier
principle. Let $X$ and $Y$ be two Banach spaces,
$\mathcal{J}:X\rightarrow \mathbb{R}$ and $\mathcal{G}:X\rightarrow
Y.$ Consider the problem
\begin{equation}\label{eq43a}
\min_{z\in X}\mathcal{J}(z) \qquad \mbox{subject to } \
\mathcal{G}(z)=0.
\end{equation}
The Lagrange function corresponding to the problem (\ref{eq43a}) is defined by
\[\mathcal{L}(z,\lambda_0,\lambda)= \lambda_0 \mathcal{J}(z)-\langle \lambda, \mathcal{G}(z)\rangle_{Y',Y}, \]
where $\lambda_0\in \mathbb{R}$ and $\lambda \in Y'$ are called
Lagrange multipliers. Then the following result is known (see e.g. \cite{gunz,iofi}).

\begin{theorem}\cite{gunz,iofi}\label{ioffe}(The Lagrange multiplier rule). Let $\hat{z}$ be a solution of (\ref{eq43a}).
Assume that the functional $\mathcal{J}$ and the mapping
$\mathcal{G}$ are continuously differentiable at the point $\hat{z}$
and that the rang of the mapping $\mathcal{G}_z(\hat{z}):X
\rightarrow Y$ is closed. Then there exists a nonzero Lagrange
multiplier $(\lambda_0,\lambda)\in \mathbb{R}^+\times Y'$, such that
\begin{eqnarray*}
\mathcal{L}_z(\hat{z},\lambda_0, \lambda)h &=&\lambda_0 \mathcal{J}_z(\hat{z})h-\langle \lambda, \mathcal{G}_z(\hat{z})h\rangle_{Y',Y}=0\quad \forall \, h \in X,\\
\mathcal{L}(\hat{z},\lambda_0, \lambda)&\leq & \mathcal{L}({z},\lambda_0, \lambda),\ \forall z\in X,
\end{eqnarray*}
where $\mathcal{L}_z(\cdot,\cdot,\cdot)$ denotes the Fr\'{e}chet
derivative of $\mathcal{L}.$ Furthermore, if
$\mathcal{G}_z(\hat{z}):X \rightarrow Y$ is an epimorphism, then
$\lambda_0\neq 0$ and $\lambda_0$ can be taken as 1.
\end{theorem}
In order to derive the first-order optimality conditions for the
problem (\ref{eq24})-(\ref{eq25}), we will apply Theorem
\ref{ioffe}.

Observing (\ref{eq25}), we define the operator
\begin{eqnarray*}
F:  \mathbb{W}_0\times L^2(D(A)') & \rightarrow &  L^2(D(A)')\\
(u,v) \ \ & \rightarrow &  F(u,v):=\Delta_\alpha u_t+\nu
\Delta_\alpha Au+ B(u,u)- v.
\end{eqnarray*}
\begin{lemma}\label{w21}
The operator $F$ is Fr\'echet differentiable with respect to $u$.
\end{lemma}
\begin{proof} Using (\ref{eq8}), $B(u+w,u+w)=B(u,u)+B(u,w)+B(w,u) +B(w,w)$.
Then
\begin{equation}\label{eq44}
F(u+w,v) - F(u,v) =\Delta_\alpha w_t+\nu \Delta_\alpha Aw +B(u,w)
+B(w,u)+B(w,w).
\end{equation}
Denoting by $Lw=\Delta_\alpha w_t+\nu \Delta_\alpha Aw +B(u,w)
+B(w,u)$, from (\ref{eq44}) we get
\begin{equation}\label{eq45}
\|F(u+w,v) - F(u,v)-Lw\|_{L^2(D(A)')}=\|B(w,w)\|_{L^2(D(A)')}.
\end{equation}
Since $w\in \mathbb{W}_0$, from (\ref{eq10}) we obtain
\[\|B(w,w)\|_{L^2(D(A)')}\leq C\|w\|_{L^\infty(V)}
\|w\|_{L^2(D(A))}\leq C\|w\|^2_{\mathbb{W}_0},\] and then from
(\ref{eq45}) we have
\[\|F(u+w,v) - F(u,v)-Lw\|_{L^2(D(A)')}\leq C\|w\|^2_{\mathbb{W}_0}.\]
Thus,
\[\lim_{\|w\|_{\mathbb{W}_0}\rightarrow 0}\frac{\|F(u+w,v)
- F(u,v)-(\Delta_\alpha w_t +\nu \Delta_\alpha
Aw+B(u,w)+B(w,u))\|_{L^2(D(A)')}} {\|w\|_{\mathbb{W}_0}}=0.\]
Therefore, the Fr\'echet derivative of $F$ with respect to $u$ in an
arbitrary $(u,v)$ is given by the operator $F_u(u,v):\mathbb{W}_0\rightarrow
L^2(D(A)')$ such that for each $w \in \mathbb{W}_0$,
\begin{equation}\label{eq46}
F_u(u,v)w=\Delta_\alpha w_t+\nu \Delta_\alpha Aw +B_u(u,u)w,
\end{equation}
where
$B_u(u,u)w=B(u,w)+B(w,u)$ is
the Fr\'echet derivative of $B$ with respect to $u$ in an arbitrary point
$(u,u)$.
\end{proof}

Now let us to consider the closed linear subspace $\mathbb{Y}_0$ of
$\mathbb{W}_0$ defined by
\[\mathbb{Y}_0=\{ w \in \mathbb{W}_0 : w(x,0)=0 \ \forall x\in \Omega\}.\]
The following preliminary result holds:
\begin{lemma}\label{l1}
Let $(u,v)\in \mathbb{W}_0\times L^2(Q)$ and $g\in L^2(D(A)')$ be
given. Then there exists a unique solution $w\in \mathbb{Y}_0$ of
the linear problem
\begin{eqnarray}\label{e1}
F_u(u, v)w=g.
\end{eqnarray}
\end{lemma}
\begin{proof} The proof follows by using the classical Galerkin approximations and
energy estimates (see \cite{BRV,real1}).
\end{proof}

\begin{lemma}
The functional $J$ is Fr\'{e}chet differentiable with respect to
$u.$
\end{lemma}
\begin{proof} From definition of the functional $J$ we get:
\begin{eqnarray*} J(u+w,v)-J(u,v)&=&\gamma_1\int_0^T (Aw,
Au-Au_d)dt + \frac{\gamma_1}{2}\int_0^T\|Aw\|^2 dt\\
&&+\gamma_2(w(T),u(T)-u_T)+\frac{\gamma_2}{2}\Vert w(T)\Vert^2.
\end{eqnarray*}
Then
\begin{eqnarray*}
&&\vert J(u+w,v)-J(u,v)- \gamma_1\int_0^T (Aw, Au-
Au_d)dt-\gamma_2(w(T),u(T)-u_T)\vert\\
&&\ \ \ \leq C (\|w\|^2_{L^2(D(A))}+\Vert w(T) \Vert^2)\leq
C\|w\|^2_{\mathbb{W}_0},
\end{eqnarray*}
which implies that
\begin{eqnarray*}
\lim_{\|w\|_{\mathbb{W}_0 \rightarrow 0}}\frac{|J(u+w,v)-J(u,v)
-\gamma_1\int_0^T (Aw,
Au-Au_d)dt-\gamma_2(w(T),u(T)-u_T)|}{\|w\|_{\mathbb{W}_0} }=0.
\end{eqnarray*}
Thus, the Fr\'echet derivative of $J$ with respect to $u$ in an
arbitrary $(u,v)$ is the operator $J_u(u,v):\mathbb{W}_0\rightarrow
\mathbb{R}$ defined by:
\begin{equation}\label{eq49}
J_u(u,v)w=\gamma_1\int_0^T (Aw, Au-Au_d)dt+\gamma_2(w(T),u(T)-u_T),\
w \in \mathbb{W}_0.
\end{equation}
With the above notations, let us define the Lagrange functional
\begin{eqnarray}\label{l2}
\mathcal{L}(u,v,\lambda)=J(u,v)-\langle F(u,v),\lambda\rangle_{L^2(D(A)'),L^2(D(A))},
\end{eqnarray}
where $\lambda\in L^2(D(A)).$ Thus, the Fr\'{e}chet derivative of
$\mathcal{L}$ with respect to $u$ is
\begin{eqnarray}\label{l3}
\mathcal{L}_u(u,v,\lambda)w=J_u(u,v)w-\langle
F_u(u,v)w,\lambda\rangle_{L^2(D(A)'),L^2(D(A))},\ \forall w\in \mathbb{W}_0.
\end{eqnarray}
\end{proof}
Now we will state and prove the necessary first-order optimality
conditions
\begin{theorem}\label{cond}(Necessary conditions) Let $(\hat{u},\hat{v})\in S_{ad}$ be a solution of the optimal control
problem (\ref{eq24})-(\ref{eq25}). Then, there exists a $\lambda\in
L^2(D(A))$ such that
\begin{eqnarray}\label{e2}
\mathcal{L}_u(\hat{u}, \hat{v},\lambda)h=J_u(\hat{u},
\hat{v})h-\langle F_u(\hat{u}, \hat{v})h,\lambda)\rangle_{L^2(D(A)'),L^2(D(A))}=0, \
\forall h\in \mathbb{Y}_0.
\end{eqnarray}
Moreover, the minimum principle holds
\begin{equation}\label{e5}
\mathcal{L}(\hat{u}, \hat{v},\lambda)\leq \mathcal{L}(\hat{u},
v,\lambda) \quad \forall v\in \mathcal{U}_{ad}.
\end{equation}
\end{theorem}
\begin{proof} We will apply Theorem \ref{ioffe}; for that, in particular, we need to prove the surjectivity of the operator $F_u(\hat{u}, \hat{v})$. Thus, in order to simplify the calculations, we rewrite the problem
(\ref{eq24})-(\ref{eq25}) in an equivalent optimal control problem.
For this purpose, by considering $z\in \mathbb{Y}_0,$ we use the
change of variable $u=\hat{u}+z.$ Thus, by replacing $u$ in
(\ref{eq25}) we get
\begin{eqnarray}\label{e3}
\Delta_\alpha z_t +\nu \Delta_\alpha Az+
B(\hat{u},z)+B(z,\hat{u})+B(z,z)=0.
\end{eqnarray}
Therefore we obtain the following equivalent optimal control problem: 
\begin{equation}\label{eq24b}
\min_{{z}\in \mathbb{Y}_0} \widetilde{J}(z):=\min_{{z}\in \mathbb{Y}_0} J(\hat{u}+z,\hat{v}),
\end{equation}
subject to the state equation
\begin{eqnarray}
 G(z)=\Delta_\alpha z_t +\nu \Delta_\alpha Az+  B(\hat{u},z)+B(z,\hat{u})+B(z,z)=0.\label{eq25b}
\end{eqnarray}
Observe that ${z}=0$ is the optimal solution of the control
problem (\ref{eq24b})-(\ref{eq25b}) provided $(\hat{u}, \hat{v})$ minimizes $J$.

Thus, we will apply Theorem \ref{ioffe} for the problem
(\ref{eq24b})-(\ref{eq25b}). For that, we will verify all its
conditions.\newline
 $-$ \textit{\textbf{Step one.} The operator $G$ is continuously
differentiable with respect to $z$.}\newline Following the proof of Lemma \ref{w21}, it is not difficult to obtain that the
derivative of the operator $G:\mathbb{Y}_0\rightarrow L^2(D(A)')$ at a point $\hat{z}$ is
given by the linear and continuous operator
$G_z(\hat{z}):\mathbb{Y}_0\rightarrow L^2(D(A)')$ defined by
\begin{eqnarray}\label{e4}
G_z(\hat{z})h=\Delta_\alpha h_t +\nu \Delta_\alpha Ah+
B(\hat{u},h)+B(h,\hat{u})+B(h,\hat{z})+B(\hat{z},h).
\end{eqnarray}
\begin{remark}\label{r1}
Notice that at the optimal solution $\hat{z}=0$ of $\widetilde{J}$ it holds
$G_z(0)h=F_u(\hat{u},\hat{v})h$ for all $h\in \mathbb{Y}_0,$ being
$(\hat{u},\hat{v}) \in S_{ad}$ the optimal solution of the control
problem (\ref{eq24})-(\ref{eq25}).
\end{remark}
$-$ {\it \textbf{Step two.} The functional $\widetilde{J}$ is
continuously differentiable with respect to $z$.}\newline Notice
that from definition of $J$ we have
\begin{eqnarray} {J}(\hat{u}+z+\epsilon h,\hat{v})-J(\hat{u}+z,\hat{v})&=&\gamma_1\epsilon\int_0^T ( Ah,
A(\hat{u}+z)-Au_d)dt + \frac{\epsilon^2\gamma_1}{2}\int_0^T\|Ah\|^2 dt\nonumber\\
&+&\epsilon\gamma_2(h(T),(\hat{u}+z)(T)-u_T)+\frac{\epsilon\gamma_2}{2}\Vert
h(T)\Vert^2.\label{w1}
\end{eqnarray}
Since
\begin{eqnarray*}
\widetilde{J}_z(z)h=\lim_{\epsilon \rightarrow
0}\frac{\widetilde{J}(z+\epsilon h)-\widetilde{J}(z)}{\epsilon}=\lim_{\epsilon \rightarrow
0}\frac{J(\hat{u}+z+\epsilon h,\hat{v})-J(\hat{u}+z,\hat{v})}{\epsilon},
\end{eqnarray*}
then, from (\ref{w1}), the derivative of $\widetilde{J}$ is given by the linear and
continuous operator $\widetilde{J}_z(z):\mathbb{Y}_0\rightarrow
\mathbb{R}$ defined by
\begin{equation}\label{eq49z}
\widetilde{J}_z(z)h=\gamma_1\int_0^T (Ah,
A\hat{u}+Az-Au_d)dt+\gamma_2(h(T),\hat{u}(T)+z(T)-u_T),\ \forall h
\in \mathbb{Y}_0.
\end{equation}
\begin{remark}\label{r2}
Notice that at the optimal solution $\hat{z}=0$ of $\widetilde{J}$ it holds
$\widetilde{J}_z(0)h=J_u(\hat{u},\hat{v})h,$ for all $h\in
\mathbb{Y}_0,$ being  $(\hat{u},\hat{v}) \in S_{ad}$ the optimal solution of the
control problem (\ref{eq24})-(\ref{eq25}).
\end{remark}

$-$\textit{\textbf{Step three.} $G_z(0):\mathbb{Y}_0\rightarrow L^2(D(A)')$ is surjective.}

Observing (\ref{e4}), the derivative of $G$ at the optimal solution
$\hat{z}=0$ of $\widetilde{J}$ is
\begin{eqnarray}
G_z(0)h=\Delta_\alpha h_t +\nu \Delta_\alpha Ah+
B(\hat{u},h)+B(h,\hat{u}).
\end{eqnarray}
Then, by using Lemma \ref{l1} together Remark \ref{r1}, for each
$g\in L^2(D(A)')$ there exists a unique $h\in \mathbb{Y}_0$ such
that
 $G_z(0)h=g.$ Notice that the range of the mapping $G_z(0):\mathbb{Y}_0\rightarrow L^2(D(A)')$ is a closed
set. Thus, the conditions of Theorem \ref{ioffe} are verified. Consecuently, by defining the Lagrange functional
\begin{equation}\label{eq51b}
\widetilde{\mathcal{L}}(z,\lambda)= \widetilde{J}(z)-\langle
G(z),\lambda\rangle_{L^2(D(A)'),L^2(D(A))},
\end{equation}
Theorem \ref{ioffe} guarantees the existence of a
$\lambda\in L^2(D(A))$ such that
\begin{equation}
\widetilde{\mathcal{L}}_z(0,\lambda)h= \widetilde{J}_z(0)h -\langle
G_z(0)h,\lambda\rangle_{L^2(D(A)'),L^2(D(A))}=0,\ \forall h\in
\mathbb{Y}_0.\label{eq51c}
\end{equation}
From Remark \ref{r1}, Remark \ref{r2}, (\ref{l3}) and
(\ref{eq51c}), we obtain
\begin{equation}\label{51d}
\widetilde{\mathcal{L}}_z(0,\lambda)h=\mathcal{L}_u(\hat{u},\hat{v})h
\ \ \forall h\in \mathbb{Y}_0,
\end{equation}
with $(\hat{u},\hat{v})\in S_{ad}$ the optimal solution of $J.$ Therefore, from (\ref{eq51c})
and (\ref{51d}), the equality (\ref{e2}) is verified. The inequality (\ref{e5}) follows directly from Theorem \ref{ioffe}.
\end{proof}
\begin{remark}
Following Theorem 1.5 in \cite{fursikov} (see also \cite{iofi}), since $\mathcal{U}_{ad}$ is a convex set, the minimum
principle (\ref{e5}) implies
\begin{equation}\label{eq56}
\mathcal{L}_v(\hat{u},\hat{v},\lambda)(v-\hat v) \geq  0 \quad
\forall \,v\in \mathcal{U}_{ad}.
\end{equation}
\end{remark}

\subsection{The weak formulation of an optimality system.}
The optimality system will be obtained from the necessary
optimality conditions given in Theorem \ref{cond}.

From (\ref{eq46}), (\ref{eq49}) and (\ref{e2}), we
obtain the adjoint equation in a weak formulation

\begin{eqnarray}
\int_0^T\langle\Delta_\alpha h_t+\nu \Delta_\alpha A h
+B_u(\hat{u},\hat{u})h,\lambda \rangle_{D(A)',D(A)} dt=\gamma_1\int_0^T (Ah, A\hat{u}-Au_d)dt&&\nonumber\\
+\gamma_2(h(T),\hat{u}(T)-u_T), \forall h\in
\mathbb{Y}_0.&&\label{eq57}
\end{eqnarray}
From the minimum principle (\ref{e5}) we have
\begin{eqnarray*}
0\leq J(\hat{u},{v})-J(\hat{u},\hat{v})+\langle
F(\hat{u},\hat{v})-F(\hat{u},{v}),\lambda \rangle_{L^2(D(A)'),
L^2(D(A))},
\end{eqnarray*}
which implies that
\begin{eqnarray}\label{e6}
0\leq \frac{\gamma_3}{2}\int_0^T(\Vert v\Vert^2-\Vert
\hat{v}\Vert^2)dt+( v-\hat{v},\lambda).
\end{eqnarray}
Using the equality $\Vert v\Vert-\Vert \hat{v}
\Vert^2=2(v-\hat{v},\hat{v})+\Vert v- \hat{v}\Vert^2,$ from
(\ref{e6}) we get
\begin{eqnarray}\label{e7bb}
0\leq \int_0^T\gamma_3 ( v-\hat{v},\hat{v}) dt+\int_0^T (
v-\hat{v},\lambda)  dt+\frac{\gamma_3}{2}\int_0^T\Vert
v-\hat{v}\Vert^2 dt.
\end{eqnarray}
From (\ref{e7bb}) we can extract the following optimality condition
\begin{eqnarray}
\int_0^T( \gamma_3\hat{v}+\lambda, v-\hat{v}) dt\geq 0  \quad
\forall\, v\in \mathcal{U}_{ad}.
\end{eqnarray}
Thus we have the variational inequality
\begin{equation}\label{eq79}
(\hat{v} + \frac{1}{\gamma_3}\lambda, v-\hat{v}) \geq 0 \quad
\mbox{a.e. in }\ Q, \, \forall\, v\in \mathcal{U}_{ad}.
\end{equation}
Moreover, since $\mathcal{U}_{ad}$ is a convex and closed set in
$L^2(Q),$ by the theorem of the projection onto a closed convex set (see
\cite{brez}), the control $\hat{v}$ in the inequality (\ref{eq79})
can be characterized as a projection; thus we have the optimality
condition
\begin{equation}\label{eq80}
\hat{v}= Proj_{\mathcal{U}_{ad}}\big( -\frac{1}{\gamma_3}
\lambda\big) \ \mbox{ a.e. in } \ Q.
\end{equation}
Consequently, the equations (\ref{eq25}), (\ref{eq57}) and the
condition (\ref{eq80}) form an optimality system
 in a weak formulation for the optimal control problem considered.
\begin{remark}
Taking into account the definition of $\mathcal{U}_{ad},$ the projection representation (\ref{eq80}) for
$\hat{v}=(\hat{v}_1,\hat{v}_2,\hat{v}_3)$ is in each component
\[\hat{v}_i=Proj_{[v_{ai}, v_{bi}]}\big( -\frac{1}{\gamma} \lambda_i\big) \ \mbox{a.e. in } \ Q, \ i=1,2,3.\]
\end{remark}

\subsection{The strong form of the optimality system}
We wish to represent the optimality system as a system of partial
differential equations with boundary, initial and terminal
conditions. Since we do not know at this moment whether the
$\lambda_t$ exists, we need to analyze the regularity of 
$\lambda.$

Using integration by parts, for $\lambda \in L^2(D(A))$
and $h\in \mathbb{Y}_0$ we get
\begin{eqnarray*}
\langle\Delta_\alpha
h_t,\lambda\rangle_{D(A)',D(A)}&=&(h_t,\lambda)+\alpha(\nabla h_t,
\nabla\lambda) =(h_t,\lambda)-\alpha(h_t, \Delta\lambda)
=(\Delta_\alpha\lambda,h_t),
\end{eqnarray*}
and then
\begin{eqnarray} \label{eq72}
\langle\Delta_\alpha h_t, \lambda\rangle_{L^2(D(A)'),L^2(D(A))}&=&
(\Delta_\alpha\lambda,h_t)_{L^2(Q),L^2(Q)}.
\end{eqnarray}

Also, by using integration by parts in (\ref{eq57}), for $ h\in
\mathbb{Y}_0$ and $ \lambda \in L^2(D(A))$, we have
\begin{eqnarray*}
\langle\nu \Delta_\alpha A h,\lambda\rangle_{D(A)',D(A)}&=&\nu(A h,
\lambda)-\alpha \nu\langle\Delta A h, \lambda\rangle_{D(A)',D(A)}
=\nu (A h, \lambda) -\alpha\nu( A h, \Delta \lambda)\\
&=&\nu (A h,\Delta_\alpha \lambda)=\langle\nu A\Delta_\alpha
\lambda, h\rangle_{D(A)',D(A)}.
\end{eqnarray*}
Then for all $h\in \mathbb{Y}_0$ we can write
\begin{equation}\label{eq58}
\langle\nu \Delta_\alpha A
h,\lambda\rangle_{L^2(D(A)'),L^2(D(A))}=\langle \nu A\Delta_\alpha
\lambda, h\rangle_{L^2(D(A)'),L^2(D(A))}.
\end{equation}
Since $B_u(\hat{u},\hat{u}):\mathbb{W}_0\rightarrow L^2(D(A)')$ and $\mathbb{Y}\subset \mathbb{W},$
then the adjoint operator of $B_u(\hat{u},\hat{u}),$ denoted by
$B_u^*(\hat{u},\hat{u}),$ is defined by:
\begin{equation}\label{eq64}
\langle B^*_u(\hat{u}, \hat{u})\lambda, h
\rangle_{\mathbb{Y}'_0,\mathbb{Y}_0}=\langle B_u(\hat{u},\hat{u}) h,
\lambda \rangle_{L^2(D(A)'),\,L^2(D(A))}.
\end{equation}
On the other hand, since $A=-P\Delta,$ we can obtain
\begin{eqnarray}
\gamma_1\int_0^T (Ah, A\hat{u}-Au_d)dt &=&-\gamma_1\int_0^T (\Delta
h, A(\hat{u}-u_d))dt\nonumber\\
&=&-\gamma_1\int_0^T \langle\Delta
A(\hat{u}-u_d),h\rangle_{D(A)',D(A)}dt.\label{e8bb}
\end{eqnarray}
Thus, from (\ref{eq57}), (\ref{eq72}), (\ref{eq58}), (\ref{eq64})
and (\ref{e8bb}), we get

\begin{eqnarray}
(\Delta_\alpha\lambda,h_t)_{L^2(Q),L^2(Q)}=-\langle\nu
A\Delta_\alpha \lambda, h\rangle_{L^2(D(A)'),L^2(D(A))}
-\langle B^*_u(\hat{u}, \hat{u})\lambda, h\rangle_{\mathbb{Y}'_0,\mathbb{Y}_0}&&\nonumber\\
-\gamma_1\langle\Delta
A(\hat{u}-u_d),h\rangle_{L^2(D(A)'),L^2(D(A))}+\gamma_2(h(T),\hat{u}(T)-u_T).&&\label{ops2}
\end{eqnarray}


In order to obtain a representation of the weak time derivative of
$\Delta_\alpha\lambda$ we analyze the regularity of $B^*_u(\hat{u},
\hat{u})\lambda$ in (\ref{ops2}).

 Notice that from (\ref{eq8}) and (\ref{eq7}) we get
\begin{eqnarray}\label{eq61}
\langle B_u(\hat{u},\hat{u})h,\lambda
\rangle_{D(A)',D(A)}&=&\langle\hat{u}\cdot\nabla \Delta_\alpha h,
\lambda\rangle_{V',V}
-\alpha((\nabla \hat{u})^*\cdot \Delta h, \lambda)\nonumber \\
&&+\langle h\cdot\nabla\Delta_\alpha\hat{u}, \lambda\rangle_{V',V}
-\alpha((\nabla h)^*\cdot \Delta\hat{u}, \lambda)\nonumber\\
&=&-(\hat{u}\cdot\nabla \lambda,\Delta_\alpha h)
-\alpha(\lambda \cdot\nabla \hat{u}, \Delta h)\nonumber\\
&&-(h\cdot\nabla \lambda,\Delta_\alpha \hat{u}) -\alpha(\lambda
\cdot\nabla h, \Delta\hat{u}).
\end{eqnarray}
We bound the terms in (\ref{eq61}). From H\"older and Sobolev inequalities we obtain
\begin{eqnarray}
|(\hat{u}\cdot\nabla \lambda,\Delta_\alpha h)|&\leq &
C\|\hat{u}\|_{L^6}\|\nabla \lambda\|_{L^3}\|\Delta_\alpha h\|
\leq C\|\hat{u}\|_V \|\lambda\|_{D(A)} \|h\|_{D(A)},\label{eq65}\\
|(\lambda \cdot\nabla \hat{u}, \Delta h)| &\leq &
C\|\lambda\|_{L^\infty}\|\nabla \hat{u}\|\| \Delta h\| \leq
C\|\lambda\|_{D(A)}\|\hat{u}\|_V \|h\|_{D(A)}.\label{eq66}
\end{eqnarray}
By observing that $ w\cdot \nabla v =0$ on $\Gamma$ if $w,v \in
D(A)$, and using integration by parts on $\Omega$, for $w, v, z \in
D(A)$ we have
\begin{equation}\label{eq67}
(w\cdot \nabla v, \Delta z)= (\nabla (w\cdot \nabla v), \nabla z)
=(\nabla v  \nabla w, \nabla z) + (w  \nabla(\nabla v), \nabla z),
\end{equation}
where $w  \nabla(\nabla v)=\sum_{i=1}^3w_i\frac{\partial}{\partial
x_i}\nabla v.$ Then, by using (\ref{eq67}), the fact that $\|\nabla
v\|_{L^4}\leq C\|v\|_{D(A)}$ and $D(A) \subset L^\infty(\Omega)$, we
obtain
\begin{eqnarray}
|(h\cdot\nabla \lambda,\Delta_\alpha \hat{u})|&=&
|(h\cdot\nabla \lambda,\hat{u})-\alpha (h\cdot\nabla \lambda,\Delta \hat{u})|\nonumber\\
&\leq & |(h\cdot\nabla \lambda,\hat{u})| +\alpha |(\nabla
\lambda\nabla h,\nabla \hat{u})|
+\alpha |(h \nabla(\nabla \lambda),\nabla \hat{u})|\nonumber \\
&\leq & C\|h\|_{L^3}\|\nabla \lambda\|\|\hat{u}\|_{L^6} +C(\|\nabla
\lambda\|_{L^4}\|\nabla h\|_{L^4}
+\|h\|_{L^\infty}\| \nabla(\nabla \lambda)\|)\|\nabla \hat{u}\|\nonumber \\
&\leq &C \|h\|_{D(A)}\|\lambda\|_{D(A)}\|\hat{u}\|_V,\label{eq68}\\
|(\lambda\cdot\nabla h,\Delta\hat{u})| &\leq & |(\nabla h\nabla
\lambda,\nabla \hat{u})|
+|( \lambda\nabla(\nabla h),\nabla \hat{u})|\nonumber \\
&\leq & C(\|\nabla h\|_{L^4}\|\nabla \lambda\|_{L^4}
+\|\lambda\|_{L^\infty}\| \nabla(\nabla h)\|)\|\nabla \hat{u}\|\nonumber\\
&\leq &C \|h\|_{D(A)}\|\lambda\|_{D(A)}\|\hat{u}\|_V.\label{eq69}
\end{eqnarray}
From (\ref{eq64}), (\ref{eq61})-(\ref{eq66}), (\ref{eq68}) and
(\ref{eq69}), and by using the H\"older inequality, for $\lambda \in
L^2(D(A))$, $\hat{u}\in \mathbb{W}_0$ and $h \in \mathbb{Y}_0$, we
have
\begin{eqnarray*}
|\langle B^*_u(\hat{u}, \hat{u})\lambda,
h\rangle_{\mathbb{Y}'_0,\mathbb{Y}_0}|\leq C
\|\hat{u}\|_{L^\infty(V)} \|\lambda\|_{L^2(D(A))}\|h\|_{L^2(D(A))},
\end{eqnarray*}
which implies
\begin{equation}\label{eq70}
B^*_u(\hat{u}, \hat{u})\lambda \in L^2(D(A)').
\end{equation}
Then, for all $h\in \mathbb{Y}_0$ we can rewrite (\ref{ops2}) as the
following equality
\begin{eqnarray*}
(\Delta_\alpha\lambda,h_t)_{L^2(Q),L^2(Q)}&=&\langle -\nu
A\Delta_\alpha \lambda - B^*_u(\hat{u}, \hat{u})\lambda -
\gamma_1\Delta
A(\hat{u}-u_d),h\rangle_{L^2(D(A)'),L^2(D(A))}\\
&&+\gamma_2(h(T),\hat{u}(T)-u_T).
\end{eqnarray*}
Since $h(T)$ can be arbitrary, when
$\Delta_\alpha\lambda(T)=\gamma_2(\hat{u}(T)-u_T),$ we have the
existence of a representation of $\Delta_\alpha \lambda_t$ in
distributional sense as being
$$\Delta_\alpha \lambda_t=\nu A\Delta_\alpha
\lambda+B^*_u(\hat{u}, \hat{u})\lambda+\gamma_1 \Delta
A(\hat{u}-u_d).$$ 
Thus we obtain that $\lambda \in L^2(D(A))$ is the
solution of
\begin{equation}
\left\{
\begin{array}[c]{rcl}
&&\Delta_\alpha\lambda_t-\nu A\Delta_\alpha \lambda-B^*_u(\hat{u},
\hat{u})\lambda=\gamma_1 \Delta A(\hat{u}-u_d)\ \mbox{in}\ L^2(D(A)'),\label{e17}\\
&&\Delta_\alpha\lambda(T)=\gamma_2(\hat{u}(T)-u_T).
\end{array}
\right.
\end{equation}
From (\ref{eq64}), (\ref{eq61}) and (\ref{eq70}) we have
\begin{eqnarray}\label{eq75}
\langle B^*_u(\hat{u}, \hat{u})\lambda,
h\rangle_{D(A)',D(A)}&=&-(\hat{u}\cdot\nabla \lambda,h)
+\alpha(\hat{u}\cdot\nabla \lambda,\Delta h)
-\alpha(\lambda \cdot\nabla \hat{u}, \Delta h)\nonumber \\
&&-(h\cdot\nabla \lambda, \Delta_\alpha \hat{u}) +\alpha(\lambda
\cdot\nabla \Delta\hat{u},h).
\end{eqnarray}
Observing that $ v\cdot \nabla w =0$ on $\Gamma$ if $v, w\in D(A)$,
and using integration by parts on $\Omega$, for $\lambda, h \in
D(A)$ we get
\begin{eqnarray}\label{eq76}
\alpha(\hat{u}\cdot\nabla \lambda,\Delta h)- \alpha (\lambda
\cdot\nabla \hat{u}, \Delta h) &=&-\alpha(\nabla(\hat{u}\cdot\nabla
\lambda),\nabla h) +\alpha(\nabla(\lambda\cdot\nabla
\hat{u}),\nabla h)\nonumber\\
&=&\alpha(\Delta(\hat{u}\cdot\nabla\lambda),
h)-\alpha(\Delta(\lambda\cdot\nabla\hat{u}), h) .
\end{eqnarray}
Taking into account (\ref{eq7}), we get
\begin{equation}\label{eq77}
-(h\cdot\nabla \lambda, \Delta_\alpha \hat{u}) =\langle h\cdot\nabla
\Delta_\alpha \hat{u},\lambda\rangle_{V',V} =-((\nabla
\lambda)^*\cdot \Delta_\alpha \hat{u},h).
\end{equation}
Thus, from (\ref{eq75})-(\ref{eq77}) we obtain
\begin{eqnarray*}
\langle B^*_u(\hat{u}, \hat{u})\lambda, h\rangle_{D(A)', D(A)}
&=&\langle-\hat{u}\cdot\nabla \lambda+\alpha\Delta(\hat{u}\cdot\nabla
\lambda)
 -\alpha \Delta(\lambda\cdot\nabla \hat{u}), h\rangle_{D(A)', D(A)}\\
&&-\langle (\nabla \lambda)^*\cdot \Delta_\alpha \hat{u}
+\alpha\lambda \cdot\nabla \Delta\hat{u},h\rangle_{D(A)', D(A)},
\end{eqnarray*}
which implies the following equality in $L^2(D(A))':$
\begin{equation}\label{eq78}
B^*_u(\hat{u}, \hat{u})\lambda=-\hat{u}\cdot\nabla \lambda
+\alpha\Delta(\hat{u}\cdot\nabla \lambda) -\alpha(\lambda\cdot\nabla
\hat{u}) -(\nabla \lambda)^*\cdot \Delta_\alpha \hat{u}
+\alpha\lambda \cdot\nabla \Delta\hat{u}.
\end{equation}

Therefore, from (\ref{e17}) and (\ref{eq78}), we have
\begin{equation}\label{eq76}
\left\{
\begin{array}{rcl}
\Delta_\alpha \lambda_t-\nu \Delta_\alpha
A\lambda+\hat{u}\cdot\nabla \lambda
&+&\alpha\Delta(\hat{u}\cdot\nabla \lambda+\lambda\cdot\nabla
\hat{u})
+(\nabla \lambda)^*\cdot \Delta_\alpha \hat{u}\\
-\alpha\lambda \cdot\nabla \Delta\hat{u}
&=&\gamma_1\Delta A(\hat{u}-u_d )\quad \mbox{ in } \ L^2(D(A)'),\\
\nabla\cdot \lambda&=& 0 \ \mbox{in } Q,\\
\lambda&=&0 \ \mbox{on } \Gamma \times (0,T),\\
\Delta_\alpha \lambda(T)&=&\gamma_2 (\hat{u}(T)-u_T)\quad \mbox{ in
} \Omega.
\end{array}
\right.
\end{equation}

Summarizing the state equation (\ref{eq25}), the adjoint equation
(\ref{eq76}) and the optimality condition (\ref{eq80}) we get the
optimality system, in the strong form, as desired.

\end{document}